\documentclass[12pt]{amsart}
\usepackage{amsmath}
\usepackage{amssymb}
\usepackage{amsrefs}
\usepackage{a4wide}
\usepackage{amscd} %AMS-LaTex package for commutative diagrams
\usepackage{graphics}
\usepackage{color}
\usepackage{amsthm} % Ams theorem style package.
\theoremstyle{plain}
\newtheorem{theorem}{Theorem}[section]
\newtheorem{lemma}[theorem]{Lemma}
\newtheorem{corollary}[theorem]{Corollary}
\newtheorem{proposition}[theorem]{Proposition}
\theoremstyle{definition}
\newtheorem{definition}[theorem]{Definition}

\theoremstyle{remark}
\newtheorem{remark}[theorem]{Remark}

\newcommand{\Ker}{\operatorname{Ker}}
\newcommand{\Rej}{\operatorname{Rej}}
\newcommand{\Rad}{\operatorname{Rad}}
\newcommand{\Jac}[1]{\mathrm{Jac}(#1)}
\newcommand{\Soc}{\operatorname{Soc}}

\newcommand{\NN}{\mathbb{N}}

\begin{document}
\title{When $\delta$-semiperfect rings are semiperfect}
\date{\today}
\author{ENG{\.I}N B\"uy\"uka\c{s}{\i}k}\address{Izmir Institute of Technology, Department of Mathematics, 35430, Urla, Izmir, Turkey}\email{enginbuyukasik@iyte.edu.tr}
\author{Christian Lomp}
\address{Departamento de Matem\'atica Pura da Faculdade de Ci\^encias da Universidade do Porto, R.Campo Alegre 687, 4169-007 Porto, Portugal}
\email{clomp@fc.up.pt}

\begin{abstract}Zhou defined $\delta$-semiperfect rings as a proper generalization of semiperfect rings. The purpose of this paper is to discuss relative notions of supplemented modules and to show that the semiperfect rings are precisely the semilocal rings which are $\delta$-supplemented. Module theoretic version of our results are obtained.
\end{abstract}
\subjclass{16D10, 16D40, 16D70}

\keywords{$\delta$-semiperfect, $\delta$-supplemented.}

\maketitle

 \renewcommand{\theenumi}{\arabic{enumi}}
\renewcommand{\labelenumi}{\emph{(\theenumi)}}

\section{Introduction}
H.\ Bass characterized in \cite{Bass} those rings $R$ whose left $R$-modules have projective covers and termed them {\it left perfect rings}. He characterized them as those semilocal rings which have a left $t$-nilpotent Jacobson radical $\Jac{R}$. Bass's {\it semiperfect rings} are those whose finitely generated left (or right) $R$-modules have projective covers. Kasch and Mares transferred in \cite{KaschMares} the notions of perfect and semiperfect rings to modules and characterized semiperfect modules by a lattice-theoretical condition as follows. A module $M$ is called {\it supplemented} if for any submodule $N$ of $M$ there exists a submodule $L$ of $M$ minimal with respect to $M=N+L$. The left perfect rings are then shown to be exactly those rings whose left $R$-modules are supplemented while the semiperfect rings are those whose finitely generated left $R$-modules are supplemented. Equivalently it is enough for a ring $R$ to be semiperfect if the left (or right) $R$-module $R$ is supplemented. Recall that a submodule $N\leq M$ is called {\it small}, denoted by $N\ll M$, if $N+L\neq M$ for all proper submodules $L$ of $M$, and that $L \leq M$, is said to be \emph{essential} in $M$, denoted by $L \unlhd M$, if $L \cap K \neq 0$ for each nonzero submodule $K \leq M$. A module $M$ is said to be \emph{singular} if $M \cong N/L$ for some module $N$ and a submodule $L\leq N$ with $L \unlhd N$.

In \cite{Zhou}, Zhou called a ring $R$ \emph{$\delta$-semiperfect} if every finitely generated $R$-module $M$ has a projective $\delta$-cover $P$, i.e. $P$ is a projective left $R$-module with a projection $p:P\rightarrow M$ onto $M$ such that the kernel $\Ker (p)$ is $\delta$-small in $P$, where a submodule $X\leq Y$ is said to be \emph{$\delta$-small} in $Y$ (denoted by $X \ll_{\delta} Y$) if $X+Z\neq Y$ for all proper $Z < Y$ with $Y/Z$ singular. It is known that ring $R$ is $\delta$-semiperfect if and only if it is a $\delta$-supplemented module. Here a module $M$ is called $\delta$-supplemented if every submodule $L\leq M$ has a $\delta$-supplement $N$ in $M$, i.e. $M=N+L$ and $N\cap L \ll_\delta N$. For further properties of $\delta$-semiperfect rings and $\delta$-supplemented modules we refer to \cite{Kosan} and \cite{Zhou}.

Zhou proved that $\delta$-semiperfect rings properly contains semiperfect rings (see, \cite{Zhou}*{Example 4.1}). An easy example of a ring that is $\delta$-semiperfect, but not semilocal had been given by Zhou in \cite{Zhou} as follows: Let $F$ be the field of two elements and $A=F^\NN$ the (commutative) ring of sequences over $F$, whose operation are pointwise multiplication and pointwise addition. Note that the unit element $1_A$ of $A$ is the sequence which is constant $1$. Let $R \subseteq A$ be the subring generated by $1_A$ and all sequences that have only a finite number of entries non-zero. Then $\Soc(R)$ consists of all sequences that have only a finite number of entries non-zero and $R/Soc(R)$ is the only singular simple $R$-module. Moreover $R/Soc(R) \simeq F$ is a field, i.e. $Soc(R)$ is an essential maximal ideal of $R$ and $R$ is $\delta$-local (see below), hence $\delta$-semiperfect. On the other hand, since $A$ is von Neumann regular, $R$ is von Neumann regular, i.e. $\Jac{R}=0$ and $R$ is not semilocal.

The purpose of this paper is to discuss the gap between supplemented and $\delta$-supplemented modules and our main result is that an arbitrary associative unital ring $R$ is semiperfect if and only if it is semilocal and $\delta$-semiperfect. We characterize finitely generated $\delta$-supplemented modules $M$ as those which are sums of simple and $\delta$-local modules or equivalently which satisfy the property that every maximal submodule of $M$ has a $\delta$-supplement. The notion of a $\delta$-coclosed submodule is defined and it is shown that a submodule is a $\delta$-supplement if and only if it is $\delta$-coclosed and a weak
$\delta$-supplement.

%\section{$\delta$-Semiperfect rings}
\section{$\delta$-supplements}

In this section we show that some of the technicalities on supplement submodules have their relative equivalent. Let $\mathcal{P}$ be the class of all singular simple $R$-modules. For a module $M$, as in \cite{Zhou}, let $$\delta (M)=\Rej (\mathcal{P})=\bigcap \{N\leq M \mid M/N \in \mathcal{P}\}=\sum \{N \leq M \mid N \ll _{\delta} M\}.$$ Let $S$ be a nonsingular simple module, then it is easy to see that $\delta(S)=S$. Also note that if $K$ is a maximal submodule which is essential in $M$, then $M/K$ is singular, so that $\delta (M) \leq K$.

We have the following basic Lemma:

\begin{lemma}[{\cite{Zhou}*{Lemma 1.2}}]\label{Lemma:Zhou-equivalent conditions for delta-small modules}
A submodule $N\leq M$ is $\delta$-small if and only if for all submodules $X \leq M$:
$$\mbox{ if } X+N=M, \mbox{ then } M=X\oplus Y \mbox{ for a projective semisimple submodule $Y$ with } Y \leq N.$$
\end{lemma}

A submodule $N\leq M$ is said to be \emph{coclosed} if $N/K \ll M/K$ implies $K=N$ for each $K\leq N$. Every supplement submodule of
a module $M$ is coclosed. The notion of coclosed submodules is generalized as follows.

\begin{definition} Let $M$ be an $R$-module and $N \leq M$. We call $N$ a \emph{$\delta$-coclosed} submodule of $M$ if $N/X$ is singular and $N/X \ll_{\delta} M/X$ for some $X \leq N,$ then $X=N$.
\end{definition}

Supplements are coclosed and so are their $\delta$-equivalents:

\begin{lemma}\label{Lemma:Delta-supplement is delta-coclosed} Let $M$ be any module and $N \leq M$ be a
$\delta$-supplement in $M$. Then $N$ is $\delta$-coclosed.
\end{lemma}

\begin{proof} Let $N$ be a $\delta$-supplement of $K$ in $M$. Then
$N+K=M$ and $N \cap K \ll_{\delta} N$. Suppose $N/X$ is singular
and $N/X \ll_{\delta} M/X$ for some $X \leq N$. Then we have
$$N/X+ (K+X)/X=M/X,$$ and $$M/(K+X) \cong N/(N\cap K+ X) $$ is
singular as a factor module of the singular module $N/X$. Therefore
we have $(K+X)/X =M/X$ as $N/X \ll_{\delta} M/X$. Then we get
$K+X=M$, and so by modular law $N=N \cap K +X$. Since $N\cap K
\ll_{\delta} N$ and $N/X$ is singular, we have $X=N$. So that $N$ is
a $\delta$-coclosed submodule of $M$.
\end{proof}

In the following proposition we give some properties of
$\delta$-coclosed submodules.

\begin{proposition}\label{Proposition:Some properties of delta-coclosed submodules} Let $N$ be a $\delta$-coclosed submodule of $M$. Then
the following hold.
\begin{enumerate}
\item[(1)] If $K\leq N \leq M$ and $K \ll _{\delta} M$ then $K \ll
_{\delta} N$. Hence $\delta (N)=N \cap \delta (M)$.

\item[(2)] If $X$ is a proper submodule of $N$ such that $N/X \ll_{\delta}
M/X,$ then $N=X \oplus X'$ for some $X' \leq N$.

\item[(3)] If $N$ is singular, then $N$ is coclosed.
\end{enumerate}

\end{proposition}

\begin{proof}$(1)$ Let $K \ll _{\delta} M$ and suppose $K+X=N$ for some $X\leq N$ with $N/X$
singular. Then $N/X=(K+X)/X \ll _{\delta}M/X$ by \cite{Zhou}*{Lemma 1.3(2)}. So that $X=N$, because $N$ is $\delta$-coclosed.

Clearly $\delta (N) \leq N \cap \delta (M)$. Therefore we only need
to prove that $N \cap \delta (M)
 \leq \delta (N)$. Let $ x \in N \cap \delta (M)$. Then $Rx \ll _{\delta} M$, and so by the first part of the proof $Rx \ll _{\delta}N$, that is, $x \in \delta
 (N)$. Hence $\delta (N)=N \cap \delta (M)$.

$(2)$ Let $X \leq N$ with $N/X \ll _{\delta} M/X$. Let $X' \leq N$
be the maximal submodule in $N$ such that $X \cap X'=0$. Then
$X\oplus X' \unlhd N$ by \cite{Anderson}*{Proposition 5.21 (1)}, and
so $N/(X\oplus X')$ is singular. On the other hand, $N/(X\oplus X')
\ll _{\delta} M/(X\oplus X')$. Since $N$ is $\delta$-coclosed, we
have $N=X\oplus X'$, as desired.

$(3)$ Since singular modules are closed under factor modules, this
is clear.
\end{proof}

\begin{corollary} Let $N$ be a $\delta$-supplement submodule of $M$.
Then $\delta (N)=N \cap \delta (M)$.
\end{corollary}

\begin{proof} By Lemma \ref{Lemma:Delta-supplement is
delta-coclosed} and Proposition \ref{Proposition:Some properties of
delta-coclosed submodules}(1).
\end{proof}

\begin{corollary}For a module $M$ and a submodule $N \leq M$, consider the following statements.

\begin{enumerate}
\item[(1)] $N$ is a $\delta$-supplement submodule of $M$.
\item[(2)] $N$ is $\delta$-coclosed in $M$.
\item[(3)] For all $X\leq N$, $X \ll _{\delta} M$ implies $X\ll
_{\delta}N$.
\end{enumerate}
Then $(1)\Rightarrow (2)\Rightarrow (3)$ hold. If $N$ has a \emph{weak $\delta$-supplement} in $M$, i.e. $N+K=M$ and $N\cap K \ll _{\delta} M$ for some submodule $K \leq M$, then $(3)\Rightarrow (1)$ holds.
\end{corollary}

\begin{proof}$(1)\Rightarrow (2)$ By Lemma \ref{Lemma:Delta-supplement is delta-coclosed}.

$(2)\Rightarrow (3)$ By Proposition \ref{Proposition:Some properties of delta-coclosed submodules}(1).

$(3)\Rightarrow (1)$ Suppose $N$ has a weak $\delta$-supplement in $M$. Then $N+L=M$ and $N\cap L \ll _{\delta}M$. Then $N\cap L \ll
_{\delta} N$ by (3), i.e $N$ is a $\delta$-supplement of $L$ in $M$.
\end{proof}

\section{On the structure of $\delta$-supplemented modules}

A module $M$ is said to be \emph{local} if $N$ has a largest proper
submodule. It is easy to see that, $M$ is local if and only if $\Rad
(M)$ is a maximal submodule of $M$ and $\Rad (M) \ll M$, (see
\cite{Wisbauer}*{41.4}).

\begin{definition} Let $M$ be an $R$-module. $M$ is said to be
$\delta$-local if $\delta (M) \ll _{\delta} M$ and $\delta (M)$ is a
maximal submodule of $M$.
\end{definition}

It is easy to see that, every simple module is local, and a simple
module is $\delta$-local if and only if it is singular. Let $S$ be a
nonsingular simple module and $S'$ be a singular simple module. Then
$S$ is local but not $\delta$-local, since $\delta (S)=S$. On the
other hand, let $M=S\oplus S'$, then clearly $M$ is not local. Since
$\delta (S)=S$ and $\delta (S')=0$, we have $\delta (M)=\delta
(S)\oplus \delta (S')=S$. Clearly $\delta (M)$ is maximal, and
nonsingularity of $S$ implies $\delta (M) \ll _{\delta}M$, so that
$M$ is $\delta$-local.

The following lemma is elementary, we include it for completeness.

\begin{lemma}\label{Lemma:local module is supplement} Let $M$ be a module and $H$ a local
submodule of $M$. Then $H$ is a supplement of each proper submodule
$K \leq M$ with $H+K=M$.
\end{lemma}

\begin{proof} Since $K$ is a proper submodule of $M$ and $K+H=M$, we
have $K\cap H$ is a proper submodule of $H$. Therefore $K\cap H \ll
H$, since $H$ is local. That is, $H$ is a supplement of $K$ in $M$.
\end{proof}

%Direct proof without contradiction:

 %Since $\Soc (M)$ is finitely generated we must have
%$\Soc (L_{k})=\Soc (L_{k+i})$ for each $i \geqslant 0$. We also have
%$L_{t}=L_{t+1}\oplus Y_{t+1}$ for some $Y_{t+1} \leq \Soc (L_{t})$
%for all $t \geqslant 1$. Hence we get $L_{k}=L_{k+i}$ for all
%$i\geqslant 0$. Formation of the $L_{i}'s$ implies that $\delta
%(L_{k})\ll L_{k}$. Therefore $L_{k}$ is a supplement of $K$ in $M$.

\begin{lemma}\label{Lemma:delta local is delta supplemented}Any $\delta$-local module is $\delta$-supplemented.
\end{lemma}

\begin{proof}Let $N \leq M$ be a proper submodule of $M$. Since $\delta (M)$ is a maximal submodule of $M$, we have either
$N\leq \delta (M)$ or $\delta (M)+N=M$. If $N \leq \delta (M)$
then, clearly $M$ is a $\delta$-supplement of $N$ in $M$. Now
suppose $N+\delta (M)=M$. Since $\delta (M)\ll _{\delta}M$, we have
by Lemma \ref{Lemma:Zhou-equivalent conditions for delta-small
modules}(2) $N\oplus Y=M$ for some semisimple submodule $Y \leq
\delta (M)$. Clearly, $Y$ is a $\delta$-supplement of $N$ in $M$.
Therefore $M$ is $\delta$-supplemented.
\end{proof}

\begin{lemma}\label{Lemma:delta supplement of maximal submodule is delta-local}Let $M$ be an $R$-module and let $K$ be a maximal
submodule with $\Soc (M)\leq K$. Suppose $L$ is a
$\delta$-supplement of $K$ in $M$, then $L$ is $\delta$-local.
\end{lemma}

\begin{proof} By hypothesis, we have $K+L=M$ and $K\cap L \ll _{\delta}
L$. We claim that $K\cap L$ is an essential submodule of $L$.
Really, if $(K\cap L)\cap T=0$ for some nonzero submodule $T \leq
L$, then $L=(K\cap L)\oplus T$ and $L/(K\cap L) \cong T$ is simple.
We get $M=K+L=K+T$, and so $T \nleq K$ gives a contradiction since
$\Soc (M) \leq K$. Therefore $\delta (L) \leq K\cap L$. Hence
$\delta (L)=
 K\cap L$
\end{proof}

A submodule $N\leq M$ is called \emph{cofinite} if $M/N$ is finitely
generated. $M$ is called \emph{cofinitely $\delta$-supplemented} if
every cofinite submodule of $M$ has a $\delta$-supplement in $M$. In
case $M$ is finitely generated, clearly every submodule of $M$ is
cofinite, and so $M$ is $\delta$-supplemented if and only if $M$ is
cofinitely $\delta$-supplemented. Therefore by \cite{KAT}*{Proposition 2.5}, if a finitely generated module $M$ is a sum of
$\delta$-supplemented modules then $M$ is $\delta$-supplemented.

\begin{proposition}\label{Proposition: Strucrure of finitely generated delta supplemented modules} For a finitely generated module $M$, the following are equivalent.

\begin{enumerate}
\item[(1)] $M$ is $\delta$-supplemented.
\item[(2)] every maximal submodule of $M$ has a $\delta$-supplement.
\item[(3)] $M=H_{1}+ H_{2} + \cdots + H_{n}$ where $H_{i}$ is either simple
or $\delta$-local.
\end{enumerate}
\end{proposition}

\begin{proof}$(1)\Rightarrow (2)$Clear.

$(2)\Rightarrow (3)$ Let $\Lambda (M) \leq M$ be the sum of all
$\delta$-supplement submodules of maximal submodules $N \leq M$
with $\Soc (M) \leq N$. Then by Lemma \ref{Lemma:delta supplement of
maximal submodule is delta-local} $\Lambda (M)$ is a sum of
$\delta$-local submodules of $M$. We claim that $M=\Soc (M) +
\Lambda (M)$. Suppose the contrary, then $\Soc (M) + \Lambda (M)\leq
K$ for some maximal submodule $K \leq M$, because $M$ is finitely
generated. By (2) $K$ has a $\delta$-supplement $L$ in $M$. Since
$\Soc (M) \leq K$, $L$ is $\delta$-local by Lemma \ref{Lemma:delta
supplement of maximal submodule is delta-local}. Hence $L \leq
\Lambda (M)\leq K$, a contradiction. Therefore $M=\Soc (M) + \Lambda
(M)$. Since $M$ is finitely generated, $M$ is a finite sum of simple
submodules and $\delta$-local submodules, as desired.

$(3)\Rightarrow (1)$ By Lemma \ref{Lemma:delta local is delta
supplemented}, $\delta$-local modules are $\delta$-supplemented, and
clearly simple modules are also $\delta$-supplemented. Therefore $M$
is $\delta$-supplemented as a finite sum of $\delta$-supplemented
modules.
\end{proof}

By \cite{Wisbauer}*{41.6}, a finitely generated module is supplemented if and only if it is a (finite) sum local modules.
Hence we can conclude from Proposition \ref{Proposition: Strucrure of finitely generated delta supplemented modules} that if any $\delta$-local submodule of a module $M$ with finitely generated socle is local, then $M$ is supplemented if and only if it is $\delta$-supplemented.

%\begin{corollary}
%Let $M$ be a finitely generated module. Suppose every
%$\delta$-local submodule of $M$ is local. Then $M$ is supplemented
%if and only if $M$ is $\delta$-supplemented.
%\end{corollary}

%Let $D\Soc (M)$ denote the submodule of $M$ which is generated byall nonsmall simple submodules of $M$. Then we have $\Soc (M)=N\oplus D\Soc(M)$ such that $N \leq \Rad (M)$ and $D\Soc (M) \cap\Rad (M)=0$.

%\begin{lemma}\label{Lemma:DSocM is finitely generated}Let $M$ be a finitely generated module. If $M/\Rad (M)$
%is semisimple, then $D\Soc (M)$ is finitely generated.
%\end{lemma}

%\begin{proof}By \cite[Proposition 2.1]{Lomp}, $M=M_{1}\oplus
%M_{2}$ with $M_{1}$ semisimple and $\Rad (M) \unlhd M_{2}$. It is
%easy to see that $M_{1}=D\Soc(M)$. Therefore $D\Soc (M)$ is finitely
%generated as a direct summand of $M$.
%\end{proof}
\section{When are $\delta$-supplemented modules supplemented}

We will turn to the problem of characterising when a $\delta$-semiperfect ring is semiperfect. Recall that a module $M$ is called semilocal if $M/\Rad(M)$ is semisimple.

\begin{lemma}\label{Lemma_Semilocal} Let $R$ be a ring and $M$ a finitely generated, $\delta$-supplemented left $R$-module.
Then $M$ is semilocal if and only if $\Soc(M)/\Soc(M)\cap \Rad(M)$ is finitely generated.
\end{lemma}

\begin{proof}
If $M$ is semilocal (and finitely generated), then $M/\Rad(M)$ is semisimple artinian. Moreover
$$ X(M)=\Soc(M)/(\Soc(M)\cap \Rad(M)) \simeq (\Soc(M)+\Rad(M))/\Rad(M) \subseteq M/\Rad(M)$$
implies $X$ to be semisimple artinian, i.e. finitely generated.

To show the converse we use induction on the length of $X(M)=\Soc(M)/(\Soc(M)\cap \Rad(M))$.
Suppose $X(M)=0$, i.e. $\Soc(M)\subseteq \Rad(M)$, then $\Rad(M)=\delta(M)$ and hence $M/\delta(M)$ is semisimple.

Assume that any finitely generated $\delta$-supplemented module $N$ with $X(N)$ of length $n\geq 0$ is semilocal and let $M$ be a finitely generated $\delta$-supplemented module with $X(M)$ having length $n+1$. Since $\Soc(M)\not\subseteq \Rad(M)$, there exists a simple direct summand $E \subseteq M$ with $M=E\oplus N$ for some $N\subseteq M$. Morever $\Rad(M)=\Rad(N)$ and $\Soc(M)=E\oplus \Soc(N)$. Hence
$$ X(M) =\Soc(M)/(\Soc(M)\cap \Rad(M)) \simeq E \oplus \Soc(N)/(\Soc(N)\cap \Rad(N) = E\oplus X(N).$$
Thus $N$ is a finitely generated $\delta$-supplemented module (direct summands of $\delta$-supplemented modules are $\delta$-supplemented) and $X(N)$ has length $n$. By induction hypothesis $N$ is semilocal and hence $M=E\oplus N$ is semilocal.
\end{proof}

It is shown in \cite{Kosan}*{Theorem 3.3} that, $\delta$-semiperfect
rings are exactly those rings $R$ that are $\delta$-supplemented as
a left (or right) $R$-module. Similarly, a ring $R$ is semiperfect
if and only if $R$ is supplemented as a left (or right) $R$-module
(see, \cite{Wisbauer}*{42.6}).

Recall that projective $\delta$-supplemented modules $M$ are $\delta$-lifting in the sense of \cite{Kosan}, i.e for every submodule $N$ of $M$ there exists a decomposition $M=D_1\oplus D_2$ such that $D_1\subseteq N$ and $N\cap D_2 \ll_\delta D_2$.

\begin{proposition}\label{Proposition_delta_lifting}
 A projective semilocal, $\delta$-supplemented module with small radical is supplemented.
\end{proposition}

\begin{proof}
Let $S=\Soc(M)=D\oplus (S\cap \Rad(M))$. Since $M$ is semilocal, there exists $N\subseteq M$ such that $D+N=M$ and $D\cap N \subseteq \Rad(M)$. But since $D\cap \Rad(M)=0$, $M=D\oplus N$ with $D$ semisimple and $\Rad(M)=\Rad(N)$. Note that $$\Soc(N)=S\cap N = (D\oplus (S\cap \Rad(M))\cap N = ((D\cap N) \oplus (S\cap \Rad(M)) = S\cap \Rad(N) \subseteq \Rad(N).$$
Hence if $K\subseteq N$ is a maximal submodule, then $N/K$ must be singular, since otherwise $N/K$ would be isomorphic to a simple direct summand of $N$ which is impossible as $\Soc(N)\subseteq \Rad(N)$. Thus $\Rad(N)=\delta(N)$. By \cite{Kosan}*{3.2} $N$ is $\delta$-lifting since it is $\delta$-supplemented and projective. Hence for any submodule $L\subseteq N$ there exist $A,B \subseteq N$ such that $N=A\oplus B$ and $A\subseteq L$ and $L\cap B \ll_\delta B$. In particular $L\cap B \subseteq \delta(B)\subseteq \delta(N)=\Rad(N)$. As $M$ has a small radical, so has $N$ and hence $N\cap B \ll N$. But since $B$ is a direct summand of $N$, $N\cap B \ll B$. This shows that $B$ is a supplement of $L$ in $N$, i.e. $N$ is a supplemented module.
We showed that $M=D\oplus N$ is the direct sum of two supplemented modules. As $M$ is projective, $M$ is itself supplemented.
\end{proof}

\begin{corollary} Let $R$ be a ring with $J=\Jac R$ and $S=\Soc(_{R}R)$. Then the following statements are equivalent.
\begin{enumerate}
\item[(a)] $R$ is semiperfect.
\item[(b)] $R$ is $\delta$-semiperfect and semilocal.
\item[(c)] $R$ is $\delta$-semiperfect and $S/S\cap J$ is finitely generated.
\end{enumerate}
\end{corollary}

\begin{proof}
$(a)\Rightarrow (b)$ is clear, $(b)\Leftrightarrow (c)$ follows from Lemma \ref{Lemma_Semilocal} and $(b)\Rightarrow (a)$ follows from Proposition \ref{Proposition_delta_lifting}.
\end{proof}

\begin{remark} In particular any ring $R$ with finitely generated left socle, e.g. $R$ left noetherian, is semiperfect if and only if it is $\delta$-semiperfect. There are $\delta$-semiperfect rings which are not semilocal and hence not semiperfect (see \cite{Zhou}*{Example 4.1}).
\end{remark}

We finish this section by showing that the last remark also holds for modules, i.e.finitely generated modules with finitely generated socle are supplemented if and only if they are $\delta$-supplemented.

\begin{lemma}\label{Lemma: SocM finitely generated then maximal delta supplement is a supplement}Let $M$ be a module and $K\leq M$ be a maximal
submodule of $M$. Suppose $\Soc (M)$ is finitely generated and $K$
has a $\delta$-supplement $H$ in $M$. Then $K$ has a supplement in
$M$ contained in $H$.
\end{lemma}

\begin{proof} By hypothesis, $H$ is a $\delta$-supplement of $K$ in $M$, that is, $K+H=M$ and $K\cap H \ll _{\delta}
H$, in particular, $K\cap H \leq \delta (H)$. Since $$M/K =(H+K)/K
\cong H/(K\cap H)$$ is simple, $K\cap H$ is a maximal submodule of
$H$. Therefore, we have either $\delta (H)=H$ or $\delta (H)= K\cap
H$. First, suppose that $\delta (H)=H$. Since $\delta (H) \ll
_{\delta} M$ (see, \cite{Zhou}*{Lemma 1.3(2)}) and $K+H=M$, we have
$K\oplus Y=M$ for a semisimple submodule $Y \leq \delta (H)$ by
\cite{Zhou}*{Lemma 1.2}. In this case, clearly $Y$ is a supplement of
$K$ in $M$.

Now, let $\delta (H)=K\cap H$. If $K\cap H \ll H$, then $H$ is a supplement of $K$ in $M$. Suppose $K\cap H=\delta (H)$
is not small in $H$, that is, $\delta (H) + L_{1}=H$ for some
proper submodule $L_{1} \lneqq H$. Then by \cite{Zhou}*{Lemma 1.2},
$H=L_{1} \oplus Y_{1}$ for some semisimple submodule $Y_{1} \leq
\delta (H)$. Since $L_{1}$ is a direct summand of $H$, we have
$$\delta (L_{1})=L_{1} \cap \delta (H)=L_{1} \cap H \cap K=L_{1}
\cap K$$ and $\delta (L_{1}) \ll _{\delta} L_{1}$. We also have
$$K+H=K +L_{1} + Y_{1}=K+L_{1}.$$ Therefore $L_{1}$ is a
$\delta$-supplement of $K$.

Since $L_{1}$ is a proper submodule of $H$ and $Y_{1}$ is a
(nonzero) semisimple module contained in $H$, we have $ \Soc (L_{1})
\lneqq \Soc (H)$. Now, if $\delta (L_{1}) \ll L_{1}$, then $L_{1}$
is a supplement of $K$ in $M$ by Lemma \ref{Lemma:local module is
supplement}, and we are done. Suppose $\delta(L_{1})$ is not small
in $L_{1}$, then $L_{1}=\delta (L_{1}) + L_{2}$ for some $L_{2}
\lneqq L_{1}$. Arguing as above we get $L_{2}$ is a
$\delta$-supplement of $K$ in $M$ with $\Soc (L_{1}) \gneqq \Soc
(L_{2})$. Continuing in this way, if non of the $L_{i}$'s is a
supplement of $K$ we shall get, a strictly descending chain of
submodules $\Soc (L_{1}) \geq \Soc (L_{2}) \geq \cdots$ of
$\Soc(M)$. This will contradict the fact that $\Soc (M)$ is finitely
generated (see, \cite{Anderson}*{Corollary 10.16}). Therefore $K$ has
a supplement in $M$.
\end{proof}

\begin{corollary}\label{Corollary:M is Delta-supplemented if and only if it is supplemented}Let $M$ be a finitely generated module. Suppose $\Soc(M)$ is finitely generated, then $M$ is supplemented if and only if
$M$ is $\delta$-supplemented.
\end{corollary}

\begin{proof}Necessity is clear. Sufficiency is a direct consequence
of Proposition \ref{Proposition: Strucrure of finitely generated
delta supplemented modules} and Lemma \ref{Lemma: SocM finitely
generated then maximal delta supplement is a supplement}.
\end{proof}

\begin{corollary}Let $M$ be a module with finitely generated socle.
Then $M$ is cofinitely supplemented if and only if $M$ is
cofinitely $\delta$-supplemented.
\end{corollary}

\begin{proof}Necessity is clear. Conversely suppose $M$ is cofinitely
$\delta$-supplemented. Let $K$ be a maximal submodule of $M$. If
$\Soc(M)$ is not contained in $K$, then we have $K+\Soc (M)=M$ by
maximality of $K$ in $M$. Then $K+S=M$ for some simple submodule of
$M$. Since $S$ is simple and $S \nleq K$, we have $K \oplus S=M$,
and hence $S$ is a supplement of $K$ in $M$. \\Now, if $\Soc (M)
\leq K$ and $H$ is a $\delta$-supplement of $K$ in $M$, then $K$ has
a supplement in $M$ by Lemma \ref{Lemma: SocM finitely generated
then maximal delta supplement is a supplement}. Hence $M$ is
cofinitely supplemented by \cite{ABS}*{Theorem 2.8}.
\end{proof}

\end{document}